\documentclass[reqno,intlimits]{amsart}
\usepackage{amssymb,mathtools,mathrsfs,procn,calrsfs,epsfig}
\usepackage[all]{xy}

\usepackage{ifpdf}
\ifpdf
 \usepackage[hyperindex]{hyperref}%,pagebackref
\else
 \expandafter\ifx\csname dvipdfm\endcsname\relax
 \usepackage[hypertex,hyperindex]{hyperref}
 \else
 \usepackage[dvipdfm,hyperindex]{hyperref}
 \fi
\fi

\def\currentvolume{1??}

\def\curryr{26}
\def\currentstartpage{?}
\def\currentendpage{??}

\numberwithin{equation}{section}

\allowdisplaybreaks[4]

\newtheorem{thm}{Theorem}[section]
\newtheorem{prop}{Proposition}[section]
\newtheorem{lem}{Lemma}[section]
\theoremstyle{remark}
\newtheorem{rem}{Remark}[section]

\DeclareMathOperator{\td}{d\!}
\DeclareMathOperator{\te}{e}

\begin{document}

\title[Monotonicity and absolute convexity of two functions]
{Monotonicity and absolute convexity of two functions involving Riemann zeta function}

\author[B.-N. Guo and F. Qi]{Bai-Ni Guo$^1$ and Feng Qi$^{2,*}$}\thanks{$^*$Corresponding author}

%\author[B.-N. Guo]{Bai-Ni Guo}
\address{$^1$17709 Sabal Court, University Village, Dallas, TX 75252-8024, USA}
\email{bai.ni.guo@gmail.com}
\urladdr{\url{https://orcid.org/0000-0001-6156-2590}}

\address{$^2$School of Mathematics and Physics, Hulunbuir University, Hulunbuir, Inner Mongolia, 021008, China}
\email{qifeng618@gmail.com}
\urladdr{\url{https://orcid.org/0000-0001-6239-2968}}

\begin{abstract}
Let $\rho>0$ be a constant, let $j\ge0$ be an integer, and let $\Gamma(z)$ denote the Euler gamma function. With the aid of the integral representation for the Riemann zeta function $\zeta(z)$, by virtue of a monotonicity rule, and by means of some properties of the function $\frac{1}{\te^t-1}$ and its derivatives, the authors discuss the increasing monotonicity of the function $t\mapsto\binom{t+\rho+j}{\rho}\frac{\zeta(t+\rho)}{\zeta(t)}$, where $\binom{z}{w}$ denotes the extended binomial coefficient, study the absolute convexity and logarithmic convexity of the function $t\mapsto\Gamma(t+j)\zeta(t)$, and derive the increasing monotonicity and inequalities of some sequences involving the ratios $\bigl|\frac{B_{2n+2}} {B_{2n}}\bigr|$ of the Bernoulli numbers $B_{2n}$.
\end{abstract}



\maketitle

\section{A concise review}

As usual and by convention, we use the following symbols
\begin{equation*}
\mathbb{Z}=\{0,\pm1,\pm2,\dotsc\}, \quad \mathbb{N}=\{1,2,\dotsc\}, \quad
\mathbb{N}_0=\{0,1,2,\dotsc\}, \quad \mathbb{N}_-=\{-1,-2,\dotsc\}.
\end{equation*}
\par
The classical Euler gamma function $\Gamma(z)$ can be defined~\cite[Chapter~6]{abram} by
\begin{equation*}
\Gamma(z)=\lim_{n\to\infty}\frac{n!n^z}{\prod_{k=0}^n(z+k)}, \quad z\in\mathbb{C}\setminus\{0,-1,-2,\dotsc\}.
\end{equation*}
See also~\cite[Chapter~3]{Temme-96-book}.
In~\cite[Fact~13.3]{Bernstein=2018-MatrixMath}, we find that, for $z\in\mathbb{C}$ such that $\Re(z)>1$, the Riemann zeta function $\zeta(z)$ is defined and satisfies
\begin{equation}\label{zeta-three-def-eq}
\zeta(z)=\sum_{k=1}^{\infty}\frac{1}{k^z}
=\frac{1}{1-2^{-z}}\sum_{k=1}^{\infty}\frac{1}{(2k-1)^z}
=\frac{1}{1-2^{1-z}}\sum_{k=1}^{\infty}(-1)^{k-1}\frac{1}{k^z}
\end{equation}
and
\begin{equation}\label{zeta-int-representation}
\zeta(z)=\frac{1}{\Gamma(z)}\int_{0}^{\infty}\frac{t^{z-1}}{\te^t-1}\td t, \quad \Re(z)>1.
\end{equation}
The last two equalities in~\eqref{zeta-three-def-eq} tell us some reasons why many mathematicians investigated the Dirichlet eta and lambda functions
\begin{equation*}
\eta(z)=\biggl(1-\frac{1}{2^{z-1}}\biggr)\zeta(z)
\quad\text{and}\quad
\lambda(z)=\biggl(1-\frac{1}{2^{z}}\biggr)\zeta(z).
\end{equation*}
According to discussions in~\cite[Section~3.5, pp.~57--58]{Temme-96-book}, the zeta function $\zeta(z)$ has an analytic continuation which has a unique singular point $z=1$, which is a simple pole with residue $1$, on the complex plane $\mathbb{C}$.
\par
We now collect several known properties and applications of the Riemann zeta function $\zeta(z)$, the Dirichlet eta function $\eta(z)$, and the Dirichlet lambda function $\lambda(z)$ as follows.
\begin{enumerate}
\item
In~\cite{kc-wang-zeta-98}, Wang proved that the eta function $\eta(t)$ is logarithmically concave on $(0,\infty)$. In~\cite{CAM-D-18-00067.tex, Rem-Bern-Ratio-Ineq.tex}, making use of Wang's result, the second author of this paper established the first double inequality for bounding the ratio $\frac{|B_{2(j+1)}|}{|B_{2j}|}$ of the Bernoulli numbers $B_{2j}$ for $j\ge1$, where the Bernoulli numbers $B_{2j}$ are generated by
\begin{equation*}
\frac{z}{\te^z-1}=\sum_{n=0}^\infty B_n\frac{z^n}{n!}=1-\frac{z}2+\sum_{j=1}^\infty B_{2j}\frac{z^{2j}}{(2j)!}, \quad \vert z\vert<2\pi.
\end{equation*}
\item
In~\cite{Cerone-Dragomie-Math-Nachr-2009}, Cerone and Dragomir proved that the reciprocal $\frac{1}{\zeta(t)}$ is concave on $(1,\infty)$.
\item
In~\cite{Zhu-Hua-JIA-2010}, Zhu and Hua proved that the sequence $\lambda(j)$ for $j\in\mathbb{N}$ is decreasing. This result was applied in~\cite{CAM-D-18-00067.tex, Rem-Bern-Ratio-Ineq.tex} to establish the first double inequality for bounding the ratio $\frac{|B_{2(j+1)}|}{|B_{2j}|}$ for $j\in\mathbb{N}$. In~\cite{Zhu-RACSAM-2020}, Zhu used this result once again to develop and sharpen Qi's first double inequality discovered in~\cite{CAM-D-18-00067.tex, Rem-Bern-Ratio-Ineq.tex}.
\item
In 2015, Adell--Lekuona~\cite{Adell-Lekuona-JNT-2015} and Alzer--Kwong~\cite{Alzer-Kwong-JNT-2015} proved the concavity of $\eta(t)$ on $(0,\infty)$.
\item
In~\cite{Hu-Kim-JMAA-2019}, Hu and Kim discovered many families of linear recurrent relations and convolution identities of $\lambda(2j)$ for $j\in\mathbb{N}$.
\item
In~\cite{Yang-Tian-JCAM-2020}, Yang and Tian verified that the function
\begin{equation*}
\frac{1}{2^t}\frac{\zeta(t)-2^{-q}\zeta(t+q)}{\zeta(t)-\zeta(t+q)}
\end{equation*}
is increasing from $(1,\infty)$ onto $\bigl(\frac{1}{2},1\bigr)$. By this, Yang and Tian~\cite{Yang-Tian-JCAM-2020} extended and sharpened Qi's first double inequality established in~\cite{CAM-D-18-00067.tex, Rem-Bern-Ratio-Ineq.tex}.
\item
In~\cite{MIA-9509.tex}, Qi and his two coauthors reviewed and surveyed some results developed in recent years about several functions involving the Riemann zeta function $\zeta(x)$ and about several sequences involving the ratio $\frac{|B_{2(j+1)}|}{|B_{2j}|}$.
\end{enumerate}
\par
As a continuation of the papers~\cite{Mon-Eta-Ratio.tex, dema-D-22-00076.tex, HJMS1099250.tex, RCSM-D-21-00302.tex}, in what follows, we will consider the following two problems.
\begin{enumerate}
\item
Define the extended binomial coefficient $\binom{z}{w}$ by
\begin{equation*}%\label{Gen-Coeff-Binom}
\binom{z}{w}=
\begin{dcases}
\frac{\Gamma(z+1)}{\Gamma(w+1)\Gamma(z-w+1)}, & z\not\in\mathbb{N}_-,\quad w,z-w\not\in\mathbb{N}_-\\
0, & z\not\in\mathbb{N}_-,\quad w\in\mathbb{N}_- \text{ or } z-w\in\mathbb{N}_-\\
\frac{\langle z\rangle_w}{w!},& z\in\mathbb{N}_-, \quad w\in\mathbb{N}_0\\
\frac{\langle z\rangle_{z-w}}{(z-w)!}, & z,w\in\mathbb{N}_-, \quad z-w\in\mathbb{N}_0\\
0, & z,w\in\mathbb{N}_-, \quad z-w\in\mathbb{N}_-\\
\infty, & z\in\mathbb{N}_-, \quad w\not\in\mathbb{Z}
\end{dcases}
\end{equation*}
in terms of the gamma function $\Gamma(z)$ and the falling factorial

\begin{equation*}%\label{falling-Factorial}
\langle\nu\rangle_m=\prod_{j=0}^{m-1}(\nu-j)
=
\begin{cases}
\nu(\nu-1)\dotsm(\nu-m+1), & m\in\mathbb{N}\\
1, & m=0
\end{cases}
\end{equation*}
for $\nu\in\mathbb{C}$.
For real number $\rho>0$ and $j\in\mathbb{N}_0$, what about the monotonicity of the function
\begin{equation}\label{zeta-ratio-gamma}
t\mapsto\binom{t+\rho+j}{\rho}\frac{\zeta(t+\rho)}{\zeta(t)}
\end{equation}
on $(1,\infty)$?
\item
What about the convexity of the function $t\mapsto\Gamma(t+j)\zeta(t)$ on $(1,\infty)$ for $j\in\mathbb{N}$?
\end{enumerate}
\par
If $(-1)^{j}g^{(j)}(t)\ge0$ for $j\ge0$ holds on an interval $I\subseteq\mathbb{R}$, then we say that $g(t)$ is a completely monotonic function on $I$; see~\cite[Chapter~XIII]{mpf-1993}, \cite[Chapter~1]{Schilling-Song-Vondracek-2nd}, and~\cite[Chapter~IV]{widder}.
If $(-1)^j[\ln g(t)]^{(j)}\ge0$ for $j\in\mathbb{N}$ holds on $I\subseteq\mathbb{R}$, then we call $g(t)$ a logarithmically completely monotonic function on $I\subseteq\mathbb{R}$; see the papers~\cite{CBerg}, \cite[Definition~1]{absolute-mon-simp.tex}, and~\cite[Definition~1]{compmon2}. If $g^{(2j)}(t)\ge0$ for $j\ge0$ holds on an interval $I\subseteq\mathbb{R}$, then we say that $g(t)$ is an absolutely convex function on $I$; if $(-1)^jg^{(2j)}(t)\ge0$ for $j\ge0$ holds on an interval $I\subseteq\mathbb{R}$, then we say that $g(t)$ is a completely convex function on $I$; see~\cite[p.~375, Definition~3]{mpf-1993}.
\par
Our main results in this paper are as follows:
\begin{enumerate}
\item
the increasing monotonicity of the function defined in~\eqref{zeta-ratio-gamma} are discussed;
\item
the absolute convexity and logarithmic convexity of the function $\Gamma(t+j)\zeta(t)$ are studied;
\item
the increasing monotonicity of a sequence involving the ratio $\frac{|B_{2n+4}|} {|B_{2n+2}|}$ and inequalities for the ratio $\frac{|B_{2n+2}|} {|B_{2n}|}$ are derived.
\end{enumerate}

\section{Preliminaries}
For proving our main results, we prepare necessary lemmas below.

\begin{lem}[{\cite[Lemma~9]{Alice-y-x-conj-one.tex} and~\cite[Remark~7.2]{Ouimet-LCM-BKMS.tex}}]\label{monotonicity-rule-qi}
Let $U(y)$, $V(y)$, $W(y,t)$ be integrable in $y\in(a,b)\subseteq\mathbb{R}$ and satisfy $V(y)>0$ and $W(y,t)>0$.
\begin{enumerate}
\item
If the ratios $\frac{\partial W(y,t)/\partial t}{W(y,t)}$ and $\frac{U(y)}{V(y)}$ are both increasing or both decreasing in $y\in(a,b)\subseteq\mathbb{R}$, then the ratio
\begin{equation*}
R(t)=\frac{\int_{a}^{b}W(y,t)U(y)\td y}{\int_{a}^{b}W(y,t)V(y)\td y}
\end{equation*}
is increasing in $t$.
\item
If one of the ratios $\frac{\partial W(y,t)/\partial t}{W(y,t)}$ and $\frac{U(y)}{V(y)}$ is increasing and another one of them is decreasing in $y\in(a,b)\subseteq\mathbb{R}$, then the ratio $R(t)$ is decreasing in $t$.
\end{enumerate}
\end{lem}

\begin{lem}[{\cite[Theorem~1.3]{Eight-Identy-More.tex}, \cite[Theorems~2.1 and~2.2]{exp-derivative-sum-Combined.tex}, \cite[Theorems~3.1 and~3.2]{CAM-D-13-01430-Xu-Cen}}] \label{Xu-Cen-rew-thm}
Let $\varphi\ne0$ and $\phi\ne0$ be real constants and $j\in\mathbb{N}$. If $\varphi>0$ and $x\ne-\frac{\ln\varphi}\phi$ or if $\varphi<0$ and $x\in\mathbb{R}$, then
\begin{equation}\label{id-gen-new-form1}
\frac{\operatorname{d}^j}{\td x^j}\biggl(\frac1{\varphi \te^{\phi x}-1}\biggr)
=(-1)^j\phi^j\sum_{q=1}^{j+1}{(q-1)!S(j+1,q)}\biggl(\frac1{\varphi \te^{\phi x}-1}\biggr)^q,
\end{equation}
where
\begin{equation*}%\label{Stirling-Number-dfn}
S(j,q)=\frac1{q!}\sum_{\ell=1}^q(-1)^{q-\ell}\binom{q}{\ell}\ell^{j}, \quad 1\le q\le j
\end{equation*}
stands for the second kind Stirling numbers.
\end{lem}

About the second kind Stirling numbers $S(m,j)$ for $m\ge j\ge0$, please see~\cite[Section~24.1.4]{abram}, \cite[Section~1.3]{Temme-96-book}, and the literature~\cite{Bell-Stirling-Lah-simp.tex, MIA-4666.tex, Quaintance-Gould-Stirling-B}.

\begin{lem}[{\cite{Dubourdieu96} and~\cite[p.~395]{haerc1}}]\label{CM-non-zero-lem}
If $G(t)$ is not identically zero and is completely monotonic on $(0,\infty)$, then $G^{(j)}(t)$ for $j\in\mathbb{N}_0$ is impossibly equal to $0$ on $(0,\infty)$.
\end{lem}

\begin{lem}\label{point-formula-thm}
When $m\in\mathbb{N}_0$, the formula
\begin{equation}\label{point-formula}
\Gamma(t+m)\zeta(t)=(-1)^m\int_{0}^{\infty}\biggl(\frac{1}{\te^x-1}\biggr)^{(m)}x^{t+m-1}\td x
\end{equation}
is valid for $\Re(t)>1$.
\end{lem}

\begin{proof}
For $m=0$, the formula~\eqref{point-formula} is just the one~\eqref{zeta-int-representation}.
\par
For $m=1$ and $\Re(t)>1$, we have
\begin{gather*}
\Gamma(t+1)\zeta(t)=t\Gamma(t)\zeta(t)
=t\int_{0}^{\infty}\frac{1}{\te^x-1}x^{t-1}\td x
=\int_{0}^{\infty}\frac{1}{\te^x-1}\frac{\td x^{t}}{\td x}\td x\\
=\frac{x^{t}}{\te^x-1}\bigg|_{x=0}^{x=\infty}-\int_{0}^{\infty}\biggl(\frac{1}{\te^x-1}\biggr)'x^{t}\td x
=-\int_{0}^{\infty}\biggl(\frac{1}{\te^x-1}\biggr)'x^{t}\td x.
\end{gather*}
\par
Assume that the formula~\eqref{point-formula} is valid for some $m\in\mathbb{N}_0$ and $\Re(t)>1$. Then
\begin{align*}
\Gamma(t+m+1)\zeta(t)&=(t+m)\Gamma(t+m)\zeta(t)\\
&=(-1)^m(t+m)\int_{0}^{\infty}\biggl(\frac{1}{\te^x-1}\biggr)^{(m)}x^{t+m-1}\td x\\
&=(-1)^m\int_{0}^{\infty}\biggl(\frac{1}{\te^x-1}\biggr)^{(m)}\frac{\td x^{t+m}}{\td x}\td x\\
&=(-1)^m\biggl(\biggl[\biggl(\frac{1}{\te^x-1}\biggr)^{(m)}x^{t+m}\biggr]\bigg|_{x=0}^{x=\infty} -\int_{0}^{\infty}\biggl(\frac{1}{\te^x-1}\biggr)^{(m+1)}x^{t+m}\td x\biggr]\biggr)\\
&=(-1)^{m+1}\int_{0}^{\infty}\biggl(\frac{1}{\te^x-1}\biggr)^{(m+1)}x^{t+m}\td x,
\end{align*}
where we used the fact that, by virtue of Lemma~\ref{Xu-Cen-rew-thm} for $\phi=\varphi=1$,
\begin{gather*}
\begin{aligned}
\biggl[\biggl(\frac{1}{\te^x-1}\biggr)^{(m)}x^{t+m}\biggr]\bigg|_{x=0}^{x=\infty}
&=(-1)^{m} \Biggl[\Biggl(\sum_{q=1}^{m+1}{(q-1)!S(m+1,q)} \biggl(\frac1{\te^x-1}\biggr)^q\Biggr) x^{t+m}\Biggr]\Bigg|_{x=0}^{x=\infty}\\
&=(-1)^{m}\sum_{q=1}^{m+1}{(q-1)!S(m+1,q)} \biggl[\biggl(\frac1{\te^x-1}\biggr)^qx^{t+m}\biggr]\bigg|_{x=0}^{x=\infty}
\end{aligned}\\
\begin{aligned}
&=(-1)^{m}\sum_{q=1}^{m+1}{(q-1)!S(m+1,q)} \biggl(\lim_{x\to\infty}\biggl[\biggl(\frac1{\te^x-1}\biggr)^qx^{t+m}\biggr] -\lim_{x\to0^+}\biggl[\biggl(\frac{x}{\te^x-1}\biggr)^qx^{t+m-q}\biggr]\biggr)\\
&=0
\end{aligned}
\end{gather*}
for $\Re(t)>1$.
Consequently, by induction, we are sure that the formula~\eqref{point-formula} is valid for all $m\in\mathbb{N}_0$ and $\Re(t)>1$. The proof of Lemma~\ref{point-formula-thm} is complete.
\end{proof}

\begin{lem}[{\cite[Theorem~1]{Exp-Diff-Ratio-Wei-Guo.tex}}]\label{F-0-G-0-lcm-thm}
For $j\in\mathbb{N}_0$, the function
\begin{equation}\label{exp=two=deriv}
\mathcal{G}_j(t)=(-1)^j\biggl(\frac{1}{\te^t-1}\biggr)^{(j)}
\end{equation}
is completely monotonic on $(0,\infty)$.
In particular, the function $\mathcal{G}_0(t)$ is logarithmically completely monotonic on $(0,\infty)$.
\end{lem}

\section{Monotonicity result and absolute convexity}

Our main results and their proofs are as follows.

\begin{thm}\label{Bernou-ratio-frac-seq-thm}
Let $\rho>0$ be a scalar and let $j\in\mathbb{N}_0$. Then 
\begin{enumerate}
\item
the function in~\eqref{zeta-ratio-gamma} for given $j\in\mathbb{N}_0$ is increasing from $(1,\infty)$ onto $(0,\infty)$;
\item
the function $\Gamma(t+j)\zeta(t)$ for given $j\in\mathbb{N}_0$ is absolutely convex in $t\in(1,\infty)$;
\item
the function $\Gamma(t+j)\zeta(t)$ for given $j\in\mathbb{N}$ is logarithmically convex in $t\in(1,\infty)$.
\end{enumerate}
\end{thm}

\begin{proof}
Making use of the formula~\eqref{point-formula} in Lemma~\ref{point-formula-thm}, we obtain
\begin{equation*}
\frac{\Gamma(t+\rho+1)\zeta(t+\rho)}{\Gamma(t+1)\zeta(t)}
=\frac{-\int_{0}^{\infty}\bigl(\frac{1}{\te^x-1}\bigr)'x^{t+\rho}\td x} {-\int_{0}^{\infty}\bigl(\frac{1}{\te^x-1}\bigr)'x^t\td x}
=\frac{\int_{0}^{\infty}\frac{\te^x}{(\te^x-1)^2}x^{t+\rho}\td x} {\int_{0}^{\infty}\frac{\te^x}{(\te^x-1)^2}x^t\td x}, \quad \Re(t)>1.
\end{equation*}
Applying Lemma~\ref{monotonicity-rule-qi} to
\begin{equation*}
U(x)=\frac{\te^x x^\rho}{(\te^x-1)^2}, \quad V(x)=\frac{\te^x}{(\te^x-1)^2}>0, \quad W(x,t)=x^t>0,
\end{equation*}
and $(a,b)=(0,\infty)$, making use of the facts that both $\frac{U(x)}{V(x)}=x^\rho$ and
\begin{equation}\label{W(x-t)-partial=ln}
\frac{\partial W(x,t)/\partial t}{W(x,t)}
=\ln x
\end{equation}
are increasing on $(0,\infty)$, we conclude that the ratio
\begin{equation*}
\frac{\int_{0}^{\infty}\frac{\te^x}{(\te^x-1)^2}x^{t+\rho}\td x} {\int_{0}^{\infty}\frac{\te^x}{(\te^x-1)^2}x^t\td x}
=\frac{\Gamma(t+\rho+1)\zeta(t+\rho)}{\Gamma(t+1)\zeta(t)}
=\Gamma(\rho+1)\binom{t+\rho}{\rho}\frac{\zeta(t+\rho)}{\zeta(t)}
\end{equation*}
is increasing in $t\in(1,\infty)$. Consequently, the function in~\eqref{zeta-ratio-gamma} for $j=0$ is increasing in $t\in(1,\infty)$.
\par
Once making use of the formula~\eqref{point-formula} in Lemma~\ref{point-formula-thm}, for $j,m\ge2$, we obtain
\begin{equation*}
\frac{\Gamma(t+\rho+j)\zeta(t+\rho)}{\Gamma(t+m)\zeta(t)}
=(-1)^{j-m} \frac{\int_{0}^{\infty}\bigl(\frac{1}{\te^x-1}\bigr)^{(j)}x^{t+\rho+j-1}\td x} {\int_{0}^{\infty}\bigl(\frac{1}{\te^x-1}\bigr)^{(m)}x^{t+m-1}\td x}
=\frac{\int_{0}^{\infty}\mathcal{G}_j(x)x^{t+\rho+j-1}\td x} {\int_{0}^{\infty}\mathcal{G}_m(x)x^{t+m-1}\td x}
\end{equation*}
for $\Re(t)\ge1$.
By Lemmas~\ref{CM-non-zero-lem} and~\ref{F-0-G-0-lcm-thm}, we see that the functions $\mathcal{G}_k(x)$ for $k\ge0$ are all positive on $(0,\infty)$.
Once applying Lemma~\ref{monotonicity-rule-qi} to
\begin{equation*}
U(x)=\mathcal{G}_j(x)x^{\rho+j}, \quad V(x)=\mathcal{G}_m(x)x^m>0, \quad W(x,t)=x^{t-1}>0,
\end{equation*}
and $(a,b)=(0,\infty)$, since $\frac{U(x)}{V(x)}=\frac{\mathcal{G}_j(x)}{\mathcal{G}_m(x)}x^{j-m+\rho}$ for $m=j$ and the partial derivative in~\eqref{W(x-t)-partial=ln} are both increasing on $(0,\infty)$, we acquire that the ratio
\begin{equation*}
\frac{\Gamma(t+\rho+j)\zeta(t+\rho)}{\Gamma(t+j)\zeta(t)}
=\Gamma(\rho+1)\binom{t+\rho+j-1}{\rho}\frac{\zeta(t+\rho)}{\zeta(t)}
=\frac{\int_{0}^{\infty}\mathcal{G}_j(x)x^{t+\rho+j-1}\td x} {\int_{0}^{\infty}\mathcal{G}_j(x)x^{t+j-1}\td x}
\end{equation*}
for $j\ge2$ and $\rho>0$ is increasing in $t\in(1,\infty)$. Consequently, the function in~\eqref{zeta-ratio-gamma} for $j\ge1$ and $\rho>0$ is increasing in $t\in(1,\infty)$.
\par
Because the ratio $\frac{\Gamma(t+\rho+j)\zeta(t+\rho)}{\Gamma(t+j)\zeta(t)}$ for given $j\in\mathbb{N}$ is increasing in $t\in(1,\infty)$, its derivative
\begin{equation*}
\biggl[\frac{\Gamma(t+\rho+j)\zeta(t+\rho)}{\Gamma(t+j)\zeta(t)}\biggr]'
=\frac{[\Gamma(t+\rho+j)\zeta(t+\rho)]'[\Gamma(t+j)\zeta(t)]
-[\Gamma(t+\rho+j)\zeta(t+\rho)] [\Gamma(t+j)\zeta(t)]'}{[\Gamma(t+j)\zeta(t)]^2}
\end{equation*}
is nonnegative in $t\in(1,\infty)$. This means that
\begin{equation*}
\frac{[\Gamma(t+\rho+j)\zeta(t+\rho)]'}{\Gamma(t+\rho+j)\zeta(t+\rho)}
\ge\frac{[\Gamma(t+j)\zeta(t)]'}{[\Gamma(t+j)\zeta(t)]},\quad t\in(1,\infty),
\end{equation*}
that is, the logarithmic derivative
\begin{equation*}
(\ln[\Gamma(t+j)\zeta(t)])'=\frac{[\Gamma(t+j)\zeta(t)]'}{[\Gamma(t+j)\zeta(t)]}
\end{equation*}
is increasing in $t\in(1,\infty)$. Consequently, for given $j\in\mathbb{N}$, the function $\Gamma(t+j)\zeta(t)$ is logarithmically convex in $t\in(1,\infty)$.
\par
Differentiating $2k$ times for $k\in\mathbb{N}_0$ with respect to $t$ on both sides of~\eqref{point-formula} yields
\begin{equation*}
[\Gamma(t+m)\zeta(t)]^{(2k)}=\int_{0}^{\infty}(-1)^m\biggl(\frac{1}{\te^x-1}\biggr)^{(m)}x^{t+m-1}(\ln x)^{2k}\td x,
\end{equation*}
where we used the dominated convergence theorem in real and functional analysis~\cite[Theorem~5.8]{Lang-3rd-RF}.
In view of Lemmas~\ref{CM-non-zero-lem} and~\ref{F-0-G-0-lcm-thm}, we see that the completely monotonic function $\mathcal{G}_m(x)=(-1)^m\bigl(\frac{1}{\te^x-1}\bigr)^{(m)}$ for $m\ge0$ is positive on $(0,\infty)$. Hence, the derivatives $[\Gamma(t+m)\zeta(t)]^{(2k)}$ for $k\in\mathbb{N}_0$ are positive, that is, they are absolutely convex on their corresponding intervals.
The proof of Theorem~\ref{Bernou-ratio-frac-seq-thm} is thus complete.
\end{proof}

\section{Monotonicity and bounds for the ratio of Bernoulli numbers}
Applying Theorem~\ref{Bernou-ratio-frac-seq-thm}, we can derive the following results on the ratio $\bigl|\frac{B_{2n+2}} {B_{2n}}\bigr|$.

\begin{thm}\label{app-thm}
The sequences $\bigl|\frac{B_{2n+2}} {B_{2n}}\bigr|$ and
\begin{equation}\label{supple-seq-mon}
\frac{(2n+j+3) (2n+j+4)}{(n+2)(2n+3)} \biggl|\frac{B_{2 n+4}}{B_{2 n+2}}\biggr|, \quad j\in\mathbb{N}
\end{equation}
are both increasing in $n\in\mathbb{N}$.
\par
The inequality
\begin{equation}\label{Qi-lower-B}
\biggl|\frac{B_{2 n+2}}{B_{2 n}}\biggr|
\ge\frac{(n+1)(2n+1)}{21}
\end{equation}
holds for $n\ge2$ and the inequality
\begin{equation}\label{Qi-lower-PPP}
\frac{B_{2 n}B_{2 n+4}}{B_{2 n+2}^2}
\ge \frac{(n+2)(2n+3)}{(n+1)(2n+1)}
\end{equation}
holds for $n\in\mathbb{N}$.
\end{thm}

\begin{proof}
In~\cite[pp.~807--808, Section~23.2]{abram} and~\cite[p.~5, (1.14)]{Temme-96-book}, we find the relation
\begin{equation}\label{Bernoulli-zeta}
B_{2n}=\frac{(-1)^{n+1}2(2n)!}{(2\pi)^{2n}}\zeta(2n), \quad n\in\mathbb{N}.
\end{equation}
Hence, from Theorem~\ref{Bernou-ratio-frac-seq-thm}, we obtain that the sequence $\binom{2n+2}{2}\frac{\zeta(2n+2)}{\zeta(2n)}$ is increasing in $n\in\mathbb{N}$, that is, the sequence
\begin{equation*}
\binom{2n+2}{2}\frac{\frac{(2\pi)^{2n+2}}{(-1)^{n+2}2(2n+2)!}B_{2n+2}} {\frac{(2\pi)^{2n}}{(-1)^{n+1}2(2n)!}B_{2n}}
=2\pi^{2}\biggl|\frac{B_{2n+2}} {B_{2n}}\biggr|
\end{equation*}
is increasing in $n\in\mathbb{N}$, that is, the ratio $\bigl|\frac{B_{2n+2}} {B_{2n}}\bigr|$ is increasing in $n\in\mathbb{N}$.
\par
Employing the relation~\eqref{Bernoulli-zeta} and Theorem~\ref{Bernou-ratio-frac-seq-thm}, we see that the sequence
\begin{align*}
\binom{2n+j+4}{2}\frac{\zeta(2n+4)}{\zeta(2n+2)}
&=\binom{2n+j+4}{2}\frac{\frac{(2\pi)^{2n+4}}{(-1)^{n+3}2(2n+4)!}B_{2n+4}} {\frac{(2\pi)^{2n+2}}{(-1)^{n+2}2(2n+2)!}B_{2n+2}}\\
&=\pi^2\frac{(2n+j+3) (2n+j+4)}{(n+2)(2n+3)} \biggl|\frac{B_{2 n+4}}{B_{2 n+2}}\biggr|
\end{align*}
for fixed $j\in\mathbb{N}$ is increasing in $n\in\mathbb{N}$.
In~\cite[Lemma~1]{JMI-4654.tex}, see also~\cite[Theorem~6]{MIA-9509.tex}, among other things, the sequence
\begin{equation*}
\frac{1} {(2n+1)(n+1)} \frac{2^{2n+2}-1} {2^{2n}-1}\biggl|\frac{B_{2n+2}}{B_{2n}}\biggr|
\end{equation*}
was proved to be increasing in $n\in\mathbb{N}$ and to tend to $\frac{2}{\pi^2}$ as $n\to\infty$. Accordingly, we arrive at
\begin{align*}
&\quad\lim_{n\to\infty}\biggl[\frac{(2n+j+3) (2n+j+4)}{(n+2)(2n+3)} \biggl|\frac{B_{2 n+4}}{B_{2 n+2}}\biggr|\biggr]\\
&=\lim_{n\to\infty}\biggl[\frac{(2n+j+3) (2n+j+4)}{(n+2)(2n+3)} (2n+3)(n+2) \frac{2^{2n+2}-1}{2^{2n+4}-1}\biggr] \\
&\quad\times\lim_{n\to\infty}\biggl[\frac{1} {(2n+3)(n+2)} \frac{2^{2n+4}-1} {2^{2n+2}-1}\biggl|\frac{B_{2 n+4}}{B_{2 n+2}}\biggr|\biggr]\\
&=\frac{2}{\pi^2}\lim_{n\to\infty}\biggl[(2n+j+3) (2n+j+4) \frac{2^{2n+2}-1}{2^{2n+4}-1}\biggr]\\
&=\infty, \quad j\in\mathbb{N}.
\end{align*}
Therefore, we are only able to acquire two one-sided inequalities
\begin{equation*}
\pi^2\frac{(2n+j+3) (2n+j+4)}{(n+2)(2n+3)} \biggl|\frac{B_{2 n+4}}{B_{2 n+2}}\biggr|
\ge \pi^{2}\frac{(j+5)(j+6)}{15}\biggl|\frac{B_6} {B_4}\biggr|
\end{equation*}
and
\begin{equation*}
\pi^2\frac{(2n+j+5) (2n+j+6)}{(n+3)(2n+5)} \biggl|\frac{B_{2 n+6}}{B_{2 n+4}}\biggr|
\ge \pi^2\frac{(2n+j+3) (2n+j+4)}{(n+2)(2n+3)} \biggl|\frac{B_{2 n+4}}{B_{2 n+2}}\biggr|
\end{equation*}
for $j\in\mathbb{N}$ and $n\in\mathbb{N}$. That is,
\begin{equation*}
\biggl|\frac{B_{2 n+4}}{B_{2 n+2}}\biggr|
\ge\frac{(n+2)(2n+3)}{15}\frac{(j+5)(j+6)}{(2n+j+3) (2n+j+4)}\biggl|\frac{B_6} {B_4}\biggr|
\end{equation*}
and
\begin{equation*}
\frac{B_{2 n+2}B_{2 n+6}}{B_{2 n+4}^2}
\ge \frac{(n+3)(2n+5)}{(n+2)(2n+3)}\frac{(2n+j+3) (2n+j+4)}{(2n+j+5)(2n+j+6)}
\end{equation*}
for $j\in\mathbb{N}$ and $n\in\mathbb{N}$. Since the sequences
\begin{equation*}
\frac{(j+5)(j+6)}{(2n+j+3) (2n+j+4)}\quad \text{and}\quad \frac{(2n+j+3) (2n+j+4)}{(2n+j+5)(2n+j+6)}
\end{equation*}
are increasing in $j\in\mathbb{N}$ and tend to $1$ as $j\to\infty$ for all fixed $n\in\mathbb{N}$ respectively, we acquire
\begin{equation*}
\biggl|\frac{B_{2 n+4}}{B_{2 n+2}}\biggr|
\ge\frac{(n+2)(2n+3)}{15}\biggl|\frac{B_6} {B_4}\biggr|
=\frac{(n+2)(2n+3)}{21}
\end{equation*}
and
\begin{equation}\label{fina-ineq}
\frac{B_{2 n+2}B_{2 n+6}}{B_{2 n+4}^2}
\ge \frac{(n+3)(2n+5)}{(n+2)(2n+3)}
\end{equation}
for $n\in\mathbb{N}$. Direct computation shows that the inequality~\eqref{fina-ineq} is also valid for $n=0$. The proof of Theorem~\ref{app-thm} is complete.
\end{proof}

\begin{rem}
The first conclusion in Theorem~\ref{app-thm} that the ratio $\bigl|\frac{B_{2n+2}} {B_{2n}}\bigr|$ is increasing in $n\in\mathbb{N}$ is a recovery of~\cite[Theorem~1.1]{RCSM-D-21-00302.tex}. The increasing monotonicity of the sequence~\eqref{supple-seq-mon} is a recovery of~\cite[Theorem~5]{MIA-9509.tex} and~\cite[Theorems~1.1 and~1.2]{RCSM-D-21-00302.tex}.
\par
The inequality~\eqref{Qi-lower-B} is better than
\begin{equation*}
\biggl|\frac{B_{2n+2}}{B_{2n}}\biggr|\ge\frac{(n+1) (2 n+1)}{30}, \quad n\in\mathbb{N}
\end{equation*}
obtained in~\cite[Remark~1]{MIA-9509.tex}.
\par
The inequality~\eqref{Qi-lower-PPP} is stronger than the logarithmic convexity of the sequence $B_{2n}$ for $n\in\mathbb{N}$, which was derived in~\cite[Theorem~1.1]{RCSM-D-21-00302.tex}.
\end{rem}

\section{A short appendix}

In this section, we slightly strengthen~\cite[Theorem~3]{Exp-Diff-Ratio-Wei-Guo.tex} as follows.

\begin{prop}\label{ratio-cm-thm}
For $j\in\mathbb{N}_0$, the ratio
\begin{equation}\label{ratios-dfn-eq}
\mathfrak{G}_j(x)=\frac{\mathcal{G}_{j+1}(x)}{\mathcal{G}_j(x)}
\end{equation}
is decreasing from $(0,\infty)$ onto $(1,\infty)$, where the function $\mathcal{G}_j(x)$ is defined by~\eqref{exp=two=deriv} in Lemma~\ref{F-0-G-0-lcm-thm}.
\end{prop}

\begin{proof}
In~\cite[Theorem~3]{Exp-Diff-Ratio-Wei-Guo.tex}, the decreasing monotonicity of the ratio $\mathfrak{G}_j(x)$ in~\eqref{ratios-dfn-eq} and $\lim_{x\to\infty}\mathfrak{G}_j(x)=1$ has been proved.
\par
Employing the equation~\eqref{id-gen-new-form1} in Lemma~\ref{Xu-Cen-rew-thm} for $\varphi=\phi=1$ yields
\begin{align*}
\mathfrak{G}_j(x)&=\frac{(-1)^{j+1}\bigl(\frac{1}{\te^x-1}\bigr)^{(j+1)}}{(-1)^j\bigl(\frac{1}{\te^x-1}\bigr)^{(j)}}\\
&=\frac{\sum_{q=1}^{j+2}{(q-1)!S(j+2,q)}\bigl(\frac1{\te^{x}-1}\bigr)^q} {\sum_{q=1}^{j+1}{(q-1)!S(j+1,q)}\bigl(\frac{1}{\te^{x}-1}\bigr)^q}\\
&=\frac{\sum_{q=1}^{j+2}{(q-1)!S(j+2,q)}\bigl(\frac{x}{\te^{x}-1}\bigr)^qx^{j-q+1}} {\sum_{q=1}^{j+1}{(q-1)!S(j+1,q)}\bigl(\frac{x}{\te^{x}-1}\bigr)^qx^{j-q+1}}\\
&\to\frac{j!S(j+2,j+1)+(j+1)!S(j+2,j+2)\lim_{x\to0^+}x^{-1}}{j!S(j+1,j+1)},\quad x\to0^+\\
&=\infty,
\end{align*}
where we considered the limits
\begin{equation*}
\lim_{x\to0^+}x^{j-q+1}=
\begin{dcases}
\lim_{x\to0^+}x^{-1},& q=j+2;\\
1,& q=j+1;\\
0, & 1\le q\le j.
\end{dcases}
\end{equation*}
The proof of Proposition~\ref{ratio-cm-thm} is thus complete.
\end{proof}

\begin{rem}
This paper is a revised version of the electronic arXiv preprint arXiv:2201.06970 at the site \url{https://doi.org/10.48550/arXiv.2201.06970}.
\end{rem}

\subsubsection*{\bf Funding}
The second author was partially supported by the Youth Project of Hulunbuir City for Basic Research and Applied Basic Research (Grant No.~GH2024020) and by the Natural Science Foundation of Inner Mongolia Autonomous Region (Grant No.~2025QN01041).

\subsubsection*{\bf Acknowledgements}
The authors appreciate the anonymous referees for their careful corrections and valuable comments on the original version of this paper.

\centerline{(Received 27.10.2025, Revised 07.02.2026, Accepted 08.06.2026)}

\end{document}